\documentclass[a4paper,11pt]{amsart}

\makeatletter
\@namedef{subjclassname@2010}{%
  \textup{2010} Mathematics Subject Classification}
\makeatother
\usepackage{a4wide}

\usepackage{eulervm}

\usepackage{amscd}
\usepackage{amsmath}
\usepackage{amssymb}
\usepackage{amsthm}
\usepackage{float}
\usepackage[dvips]{graphicx}
\usepackage{enumerate}

\usepackage{array}
\usepackage{amscd}
\usepackage[all]{xy}
\usepackage{graphicx}
\usepackage{url}
\usepackage{enumitem}

\usepackage{amsfonts}
\usepackage{amssymb}
\usepackage{amsmath}
\usepackage{amsthm}

\usepackage{mathdots}

\usepackage[all]{xy}
\usepackage{graphicx}
\usepackage{mathrsfs}
\usepackage{color}
\usepackage{url}
\usepackage{array}
\usepackage{pifont}
\usepackage{tikz}
\usepackage{caption}
\usepackage{booktabs}
\usepackage{setspace}

\DeclareFontFamily{U}{rsfs}{%
\skewchar\font127}
\DeclareFontShape{U}{rsfs}{m}{n}{%
<-6>rsfs5<6-8.5>rsfs7<8.5->rsfs10}{}
\DeclareSymbolFont{rsfs}{U}{rsfs}{m}{n}
\DeclareSymbolFontAlphabet
{\mathrsfs}{rsfs}
\DeclareRobustCommand*\rsfs{%
\@fontswitch\relax\mathrsfs}

\theoremstyle{plain}
\newtheorem{thm}{Theorem}[section]
\newtheorem{prop}[thm]{Proposition}
\newtheorem{lem}[thm]{Lemma}
\newtheorem{defi}[thm]{Definition}
\newtheorem{rmk}[thm]{Remark}

\newtheorem{nota}[thm]{Notation}

\newtheorem{prop-defi}[thm]{Proposition-Definition}
\newtheorem{thm-defi}[thm]{Theorem-Definition}
\newtheorem{lem-defi}[thm]{Lemma-Definition}

\newtheorem{conj}[thm]{Conjecture}

\newdimen\argwidth
\def\db[#1\db]{
 \setbox0=\hbox{$#1$}\argwidth=\wd0
 \setbox0=\hbox{$\left[\box0\right]$}
  \advance\argwidth by -\wd0
 \left[\
 n.3\argwidth\box0 \kern.3\argwidth\right]}

\newcommand{\SX}{\scriptscriptstyle  X}
\newcommand{\SY}{\scriptscriptstyle  Y}
\newcommand{\HX}{\widehat{X}}
\newcommand{\HY}{\widehat{Y}}
\newcommand{\SHX}{\scriptscriptstyle  \widehat{X} }
\newcommand{\SHY}{\scriptscriptstyle  \widehat{Y}}
\newcommand{\HL}{\widehat{L}}
\newcommand{\lx}{\ell_{\scriptscriptstyle{X}} }
\newcommand{\ly}{\ell_{\scriptscriptstyle{Y}}}
\newcommand{\lhx}{\ell_{\scriptscriptstyle  \widehat{X}} }
\newcommand{\lhy}{\ell_{\scriptscriptstyle  \widehat{Y}}}
\newcommand{\SH}{\operatorname{\scriptscriptstyle{H}}}
\newcommand{\calExt}{\operatorname{\mathcal{E}\textit{xt}}}
\newcommand{\Adiag}{\operatorname{Adiag}}
\newcommand{\alg}{\operatorname{\scriptstyle{alg}}}
\newcommand{\calHom}{\operatorname{\mathcal{H}\textit{om}}}

\newcommand{\BL}{\operatorname{\mathbf{L}}}
\newcommand{\HPsi}{\widehat{\Psi}}

\newcommand{\aA}{\mathcal{A}}
\newcommand{\bB}{\mathcal{B}}
\newcommand{\cC}{\mathcal{C}}
\newcommand{\CC}{\mathbb{C}}
\newcommand{\dD}{\mathcal{D}}
\newcommand{\eE}{\mathcal{E}}
\newcommand{\fF}{\mathcal{F}}

\newcommand{\oO}{\mathcal{O}}
\newcommand{\pP}{\mathcal{P}}

\newcommand{\QQ}{\mathbb{Q}}

\newcommand{\RR}{\mathbb{R}}
\newcommand{\sS}{\mathcal{S}}
\newcommand{\tT}{\mathcal{T}}

\newcommand{\ZZ}{\mathbb{Z}}

\newcommand{\Supp}{\mathop{\rm Supp}\nolimits}
\newcommand{\Hom}{\mathop{\rm Hom}\nolimits}

\newcommand{\dR}{\mathbf{R}}

\newcommand{\NS}{\mathop{\rm NS}\nolimits}

\newcommand{\Pic}{\mathop{\rm Pic}\nolimits}

\newcommand{\ch}{\mathop{\rm ch}\nolimits}
\newcommand{\rk}{\mathop{\rm rk}\nolimits}

\newcommand{\Coh}{\mathop{\rm Coh}\nolimits}

\newcommand{\cneq}{\mathrel{\raise.095ex\hbox{:}\mkern-4.2mu=}}
\newcommand{\eqcn}{\mathrel{=\mkern-4.5mu\raise.095ex\hbox{:}}}

\newcommand{\End}{\mathop{\rm End}\nolimits}

\newcommand{\Imm}{\operatorname{Im}}

\newcommand{\Ree}{\operatorname{Re}}

\newcommand{\cl}{\mathop{\rm cl}\nolimits}

\begin{document}

\title[Fourier-Mukai transforms and stability conditions]{Fourier-Mukai transforms and stability conditions on Abelian Varieties}

\date{\today}

\author{Dulip Piyaratne}

\address{Kavli Institute for the Physics and Mathematics of the Universe (WPI)\\ The University of Tokyo Institutes for Advanced Study \\ The University of Tokyo \\ Kashiwa \\ Chiba 277-8583 \\ Japan.}

\email{dulip.piyaratne@ipmu.jp}



\begin{abstract}
This article is based on a talk given at the  Kinosaki Symposium on Algebraic Geometry in 2015, about a 
work in progress \cite{Piy2}. We describe a polarization on a derived equivalent abelian variety by using Fourier-Mukai theory. We explicitly formulate a conjecture which says certain Fourier-Mukai transforms between derived categories give equivalences of some hearts of Bridgeland stability conditions. We establish it for abelian surfaces, which is already known due to D. Huybrechts, and for abelian 3-folds. This generalizes the author's previous joint work \cite{MP1, MP2}  with A. Maciocia on principally polarized abelian 3-folds with Picard rank one. Consequently, we see that the strong Bogomolov-Gieseker type inequalities hold for tilt stable objects on abelian 3-folds. 
\end{abstract}

\maketitle
%%%%%%%%%%%%%%%%%%%%%%%%%%%%%%%%%%
%%%%%%%%%%%%%%%%%%%%%%%%%%%%%%%%%%
\section{Introduction}
%\subsection{Background}
%%%%%%%%%%%%%%%%%%%%%%%%%%%%%%%%%%
%%%%%%%%%%%%%%%%%%%%%%%%%%%%%%%%%%
The notion of Fourier-Mukai transform  was introduced by
Mukai in early 1980s (see \cite{Muk2}). In particular, he showed that the Poincar\'e bundle induces a
non-trivial equivalence between the derived categories of an abelian variety and its
dual variety. 
Furthermore, he studied certain type of vector bundles on abelian varieties called 
semihomogeneous bundles, and moduli of them (see \cite{Muk1}).
In particular, the moduli space parametrizing simple semihomogeneous bundles on an abelian variety $Y$ with a fixed Chern character is also an abelian variety, denoted by $X$. Moreover, the associated universal bundle $\eE$ on $X \times Y$ induces a derived equivalence $\Phi_{\eE}^{\SX \to \SY}$ from $X$ to $Y$, which is commonly known as the Fourier-Mukai transform. 
This transform induces a linear isomorphism $\Phi_{\eE}^{\SH}$ from
$H^{2*}_{\alg}(X,\QQ)$ to $H^{2*}_{\alg}(Y,\QQ)$, called the cohomological Fourier-Mukai transform.
In this article, we realize this linear isomorphism in anti-diagonal form with respect to some twisted Chern characters (see  Theorem \ref{antidiagonalrep}). This generalizes  similar work on principally polarized abelian varieties with Picard rank one case (see \cite{MP2}).
The following is one of the important consequences of this anti-diagonal representation, which generalizes the known result \cite{BL1} for the classical Fourier-Mukai transform with kernel the Poincar\'e bundle.
\begin{thm}[\,= \ref{polarize}\,]
For  $a \in X$, $b \in Y$, let  $\Xi : D^b(X) \to D^b(Y)$ be the Fourier-Mukai functor defined by 
$$
\Xi = \eE_{\{a\}\times Y}^* \circ \Phi_\eE^{\SX  \to \SY} \circ \eE_{X \times \{b\}}^*,
$$
where $\eE_{\{a\}\times Y}$ denotes the functor $\eE_{\{a\}\times Y} \otimes (-)$ and similar for $ \eE_{X \times \{b\}}$. 
If the ample line bundle $L$ defines a polarization on $X$,  then
the ample line bundle $\det (\Xi(L))^{-1}$ defines a polarization on $Y$.
\end{thm}

Motivated by Douglas's work on $\Pi$-stability for D-branes, Bridgeland introduced the notion of stability conditions on triangulated categories (see \cite{BriStab}). This categorical stability notion  can be interpreted essentially as an abstraction of the usual slope stability for sheaves on projective curves. 
The category of coherent sheaves does not arise as a heart of a Bridgeland stability condition for higher dimensional smooth projective varieties. So more work is needed to construct the hearts for stability conditions on projective varieties of dimension above one. In general, when $B +i\omega$ is a complexified ample class on a projective variety $X$, it is expected that 
$Z_{B+i\omega}^{\SX}(-) = -\int_X e^{-B - i\omega}\ch(-)$ defines 
a central charge function of some stability condition on $X$. 
In this article we conjecturally construct a heart for this central charge function by using the  notion of
 very weak stability condition (see Conjecture \ref{conjstab}).  This essentially generalizes  the  single tilting construction of Bridgeland and Arcara-Bertram for surfaces, and the conjectural double tilting construction of Bayer-Macr\`i-Toda for 3-folds.

Action of the Fourier-Mukai transform $\Phi_{\eE}^{\SX \to \SY}$ induces stability conditions on $D^b(Y)$ from the ones on $D^b(X)$.
This can be defined via the induced the map on $\Hom(K(Y), \CC)$ from $\Hom(K(X), \CC)$ by the transform.
For abelian varieties we view this as  
$\Phi_{\eE}^{\SX \to \SY} \cdot Z_{\Omega}^{\SX} = \zeta \, Z_{\Omega'}^{\SY}$ 
for some $\zeta \in \CC$, where $\Omega, \Omega'$ are complexified ample classes on $X, Y$ respectively.  When $\zeta$ is real one can expect that the Fourier-Mukai transform $\Phi_{\eE}^{\SX \to \SY}$   gives an equivalence of some hearts of particular stability conditions on $X$ and $Y$, whose $\Omega$ and $\Omega'$ are determined by $\Imm \zeta  = 0$. More precisely we formulate the following:

\begin{conj}[\,= \ref{conjequivalence}\,]
\label{conjequiintro}
The Fourier-Mukai transform $\Phi_{\eE}^{\SX \to \SY}: D^b(X) \to D^b(Y)$ between the abelian varieties gives the equivalence of stability condition hearts each of which are $g$-fold tilts as  in Conjecture \ref{conjstab}:
$$
\Phi_{\eE}^{\SX \to \SY} [k]  \left(  \aA^{\SX}_{\Omega_k}  \right) = \aA^{\SY}_{\Omega'_k}.
$$
Here  $\Omega_k, \Omega_k'$
are some complexified ample classes on $X, Y$ respectively,  $1 \le k \le  (g-1)$, and $g= \dim X = \dim Y$.
\end{conj}

In the last section we establish the above conjecture for abelian surfaces and abelian 3-folds. In particular, we explicitly give the details for the surface case and our proof is essentially a generalization of Yoshioka's in \cite{Yos}, and closely follow the one in the author's PhD thesis \cite[Chapter 6]{Piy}. The equivalences of such stability condition hearts are already known for a principally polarized abelian 3-folds with Picard rank one. Since we have the anti-diagonal representation of cohomological Fourier-Mukai transform as in Theorem \ref{antidiagonalrep}, essentially the same arguments work in the general case. In particular, we deduce the following:
\begin{thm}[\,= \ref{equivalence3fold}\,]
Conjecture \ref{conjequiintro} is true for abelian 3-folds. Moreover, the strong Bogomolov-Gieseker type inequality introduced in \cite{BMT} holds for abelian 3-folds.
\end{thm}

%%%%%%%%%%%%%%%%%%%%%%%%%%%%%%%%%%
%%%%%%%%%%%%%%%%%%%%%%%%%%%%%%%%%%
\section*{Notation}
%%%%%%%%%%%%%%%%%%%%%%%%%%%%%%%%%%
%%%%%%%%%%%%%%%%%%%%%%%%%%%%%%%%%%
%Some of the important notations in this paper are as follows:

\begin{itemize}[leftmargin=*]
\item Unless otherwise stated,  throughout this paper, all the varieties are smooth projective and defined over 
$\mathbb{C}$.  For a variety $X$, by $\Coh(X)$ we denote  the category of 
coherent sheaves on $X$, and by $D^b(X)$ we denote the bounded derived category of $\Coh(X)$.

\item For $0 \le i \le \dim X$, $\Coh_{\le i}(X) : = \{E \in \Coh(X): \dim \Supp(E)  \le i  \}$, $\Coh_{\ge i}(X) := \{E \in \Coh(X): \text{for } 0 \ne F \subset E,   \ \dim \Supp(F)  \ge i  \}$ and $\Coh_{i}(X) : = \Coh_{\le i}(X) \cap \Coh_{\ge i}(X)$. 

\item For $E \in D^b(X)$, $E^\vee = \dR\calHom(E, \oO_X)$. When $E$ is a torsion free sheaf we write its dual by $E^*$; so for a locally free sheaf $E$, $E^\vee  = E^*$.

\item The structure sheaf of a closed subscheme $Z \subset X$ as an object in $\Coh(X)$ is denoted by $\oO_Z$ and when $Z = \{x\}$ for some closed point $x\in X$ it is simply denoted by $\oO_x$.

\item When  $\aA$ is the heart of a bounded t-structure  on 
a triangulated category $\dD$,  by 
$H_{\aA}^i(-)$ we denote the corresponding  $i$-th cohomology functor. 
\item 
For a set of objects $\sS \subset \dD$ in a triangulated category $\dD$,  by 
$\langle \sS \rangle \subset \dD$ we denote
its extension closure, that is the smallest extension closed subcategory 
of $\dD$ which contains $\sS$.

\item We denote the upper half plane $\{z \in \mathbb{C} : \Imm z >0\}$ by $\mathbb{H}$.

\item We will denote an $m \times m$ anti-diagonal matrix with entries $a_k$, $k=1, \ldots,  m$ by
$$
\Adiag(a_1, \ldots, a_m)_{ij} : = \begin{cases}
a_k & \text{if } i=k, j=m+1-k \\
0 & \text{otherwise}.
\end{cases}
$$

\end{itemize}
\section{Preliminaries}
%%%%%%%%%%%%%%%%%%%%%%%%%%%%%%%%%%
%%%%%%%%%%%%%%%%%%%%%%%%%%%%%%%%%%

%%%%%%%%%%%%%%%%%%%%%%%%%%%%%%%%%%
%%%%%%%%%%%%%%%%%%%%%%%%%%%%%%%%%%
\subsection{Very weak stability conditions}\label{subsec:vw}
%%%%%%%%%%%%%%%%%%%%%%%%%%%%%%%%%%
%%%%%%%%%%%%%%%%%%%%%%%%%%%%%%%%%%
Let us recall the 
general arguments of very weak stability conditions, 
which are the variants of weak stability of~\cite{Tcurve1}
introduced in~\cite[Definition~B.1]{BMS}.
We closely follow the notions as in \cite{PT}.

Let $\dD$ be a triangulated category, and 
$K(\dD)$ its Grothendieck group. 
We fix a finitely generated free abelian 
group $\Gamma$ and a group homomorphism
\begin{align*}
\cl \colon K(\dD) \to \Gamma. 
\end{align*}

\begin{defi}
\rm
A \textit{very weak stability condition}
on $\dD$ is a pair $(Z, \aA)$, 
where $\aA$ is the heart of a bounded t-structure on $\dD$, 
and $Z \colon \Gamma \to \mathbb{C}$ is a group homomorphism 
satisfying the following conditions: 
\begin{enumerate}
\item For any $E \in \aA$, we have
$Z(E) \in \mathbb{H} \cup \mathbb{R}_{\le 0}$.
Here and elsewhere we write $Z(E) = Z(\cl(E))$, and $\mathbb{H}$ is the 
upper half plane. 

\item The associated slope function $\mu:  \aA \to \mathbb{R} \cup \{+\infty\}$ is defined by 
$$
\mu (E) =
\begin{cases}
+ \infty & \ \text{if } \Imm Z(E)=0 \\
 -\frac{\Ree Z(E)}{\Imm Z(E)} & \ \text{otherwise},
\end{cases}
$$
and it satisfies the Harder-Narasimhan  property.  
\end{enumerate}
\end{defi}

We say that 
$E \in \aA$ is $\mu$-(semi)stable 
if for any non-zero subobject $F \subset E$
in $\aA$, 
we have the inequality:
$\mu(F) <(\le) \, \mu(E/F)$.

The  Harder-Narasimhan filtration of an object $E \in \aA$ is 
a chain of subobjects 
$0=E_0 \subset E_1 \subset \cdots \subset E_n=E$
 in $\aA$ such that each $F_i=E_i/E_{i-1}$ is 
$\mu$-semistable with 
$\mu(F_i)>\mu(F_{i+1})$. 
If such  Harder-Narasimhan filtrations exists for all objects in $\aA$,
we say that $\mu$ satisfies the  Harder-Narasimhan property.

For a given a very weak stability condition $(Z, \aA)$, 
we define its slicing on $\dD$ (see \cite[Definition~3.3]{BriStab})
\begin{align*}
\{\pP(\phi)\}_{\phi \in \mathbb{R}}, \
\pP(\phi) \subset \dD
\end{align*}
as in the case of Bridgeland 
stability conditions (see \cite[Proposition~5.3]{BriStab}). 
Namely, for $0<\phi \le 1$, 
the category $\pP(\phi)$ is defined to 
be 
\begin{align*}
\pP(\phi) =\{ 
E \in \aA : 
E \mbox{ is } \mu \mbox{-semistable with }
\mu(E)=-1/\tan (\pi \phi)\} \cup \{0\}.
\end{align*} 
Here we set $-1/\tan \pi =\infty$. 
The other subcategories are defined by setting
\begin{align*}
\pP(\phi+1)=\pP(\phi)[1].
\end{align*}
For an interval $I \subset \mathbb{R}$, 
we define $\pP(I)$ to be the smallest extension 
closed subcategory of $\dD$ which contains 
$\pP(\phi)$ for each $\phi \in I$. 
For $0 \le s \le 1$, the pair $(\pP((s, 1]), \pP((0, s]) )$ of subcategories of $\aA = \pP((0,1])$ is a torsion pair, and the corresponding tilt is $\pP((s, s+1])$. 

Note that  
the category $\pP(1)$ contains the 
following category 
\begin{align*}
\cC \cneq \{E \in \aA : Z(E)=0\}. 
\end{align*}
It is easy to check that $\cC$ 
is closed under subobjects and quotients in $\aA$. 
In particular, $\cC$ is an 
abelian subcategory of 
$\aA$. 
Moreover, 
the pair $(Z, \aA)$ gives a \textit{Bridgeland stability 
condition} on $\dD$ if $\cC=\{0\}$. 
%%%%%%%%%%%%%%%%%%%%%%%%%%%%%%%%%%
%%%%%%%%%%%%%%%%%%%%%%%%%%%%%%%%%%
\subsection{Bridgeland stability on varieties}
%%%%%%%%%%%%%%%%%%%%%%%%%%%%%%%%%%
%%%%%%%%%%%%%%%%%%%%%%%%%%%%%%%%%%
Let $X$ be a  smooth projective variety and let $D^b(X)$  be
the bounded derived category of coherent sheaves on $X$.
A stability condition $\sigma$ on $D^b(X)$ is
\textit{numerical} if the central charge function $Z: K(X) \to \CC$ factors through the Chern character map
$\ch: K(X) \to H^{2*}_{\alg}(X,\QQ)$.
%$$
%\xymatrixcolsep{6pc}
%\xymatrixrowsep{4pc}
%\xymatrix{
%K(X) \ar[r]^{Z} \ar[d]^{\ch} & \CC  \\
%H^{2*}_{\alg}(X,\QQ) \ar@{-->}[ru]
%}.
%$$
%A numerical stability condition $\sigma$ satisfies the \textit{support property}
%if there is a positive constant $C$ such that for any $\sigma$-stable object
%$E$, we have
%$||\ch(E)|| \le C |Z(E)|$.
%Here $||-||$ is  a fixed norm on the vector space $ H^{2*}_{\alg}(X,\QQ)$.

Usually, a stability condition $\sigma$ on $D^b(X)$ is called \textit{geometric} if
all the skyscraper sheaves $\oO_x$ of $x \in X$ are $\sigma$-stable of the same phase.
The following result gives some properties of  geometric stability
conditions on  varieties.

\begin{prop}
Let $X$ be a smooth projective variety of dimension $n$.
Let $\sigma= (Z, \aA)$ be a geometric stability condition on $D^b(X)$ with
all skyscraper sheaves $\oO_x$ of $x \in X$ are $\sigma$-stable with phase one.
If $E \in  \aA$ then $H^i_{\Coh(X)}(E) = 0$ for $i \notin \{-n+1, -n+2, \ldots ,0\}$.
\end{prop}
\begin{proof}
The following proof is adapted from \cite[Lemma 10.1]{BriK3}.
Let $\pP$ be the corresponding slicing of $\sigma$.
Since $\aA = \pP((0,1])$ and $\Coh_0(X) \subset \pP(1)$, from the Harder-Narasimhan property, we only need to consider $E \in \aA $ such that $\Hom_{\SX} (\Coh_0(X), E) =0$.
For any skyscraper sheaf $\oO_x$ of $x \in X$ we have $\oO_x[i] \in \pP(1+i)$ and  $E[i] \in \pP((i,1+i])$.
Therefore, for all $i<0$,
$\Hom_{\SX}(E, \oO_x[i]) = 0$, and $\Hom_{\SX}(\oO_x, E[1+i])$ $\cong$ $\Hom_{\SX}(E,\oO_x [n-1-i])^* =0$.
So by \cite[Proposition 5.4]{BM}, $E$ is quasi-isomorphic to a complex of locally free sheaves of length $n$.
This completes the proof as required. 
\end{proof}

When $X$ is a smooth projective curve, the central charge  function 
$Z(E) = -\deg(E) + i \rk (E)$ together with the  heart 
$\Coh(X)$ of the standard t-structure defines a geometric stability condition on $D^b(X)$.
However, for a smooth projective variety $X$ with $\dim  X  \ge 2$, there is no numerical
stability condition on $D^b(X)$ with $\Coh(X)$ as the heart of a stability condition (see \cite[Lemma 2.7]{TodLimit} for a proof). 
Motivated by the constructions for smooth projective surfaces  (see \cite{BriK3, AB}) together with some observations in  Mathematical Physics, for any dimensional smooth projective variety $X$, it is expected that 
the function defined by
\begin{equation*}
Z_{B + i \omega}(-) = - \int_X e^{-B- i \omega} \ch(-)
\end{equation*}
is a central charge function of some geometric stability condition on $D^b(X)$ (see \cite[Conjecture 2.1.2]{BMT}). 
Here $B + i \omega \in \NS_{\CC}(X)$ is a complexified ample class on $X$, that is by definition $B, \omega \in \NS_{\RR}(X)$ with $\omega$ an ample class. 
By using the notion of very weak stability, let us conjecturally construct a heart for this central charge function.  

Let $n=\dim X $. For $0\le k\le n$, we define the $k$-truncated Chern character by
$$
\ch_{\le k}(E) = (\ch_0(E), \ch_1(E), \ldots, \ch_k(E), 0, \ldots, 0),
$$
and the function $Z^{(k)}_{B + i \omega} : K(X) \to \CC$ by
\begin{equation*}
Z^{(k)}_{B + i \omega}(E) = - i^{n-k}\int_X e^{-B- i \omega} \ch_{\le k}(E).
\end{equation*}

The usual slope stability on sheaves gives the very weak stability condition $(Z^{(1)}_{B + i \omega}, \Coh(X))$. 
Moreover, we formulate the following:
\begin{conj}[\cite{Piy2}]
\label{conjstab}
For each $1 \le k < n$, the pair $\sigma_k = (Z^{(k)}_{B + i \omega}, \aA^{(k)}_{B + i \omega})$ gives a very weak stability condition on $D^b(X)$, where the hearts $\aA^{(k)}_{B + i \omega}$, $1\le k \le n$  are defined by 
\begin{equation*}
 \left.\begin{aligned}
         & \aA^{(1)}_{B + i \omega} = \Coh(X) \\
         &\aA^{(k+1)}_{B + i \omega} = \pP_{\sigma_k}((1/2,\,3/2])
       \end{aligned}
  \ \right\}.
\end{equation*}
Moreover, the pair $\sigma_n = (Z_{B + i \omega}, \aA^{(n)}_{B + i \omega})$ is a Bridgeland stability condition on $D^b(X)$. 
\end{conj}

\begin{rmk}
\label{remarkconjstab}
\rm
As  mentioned before, $(Z^{(1)}_{\omega,B}, \aA^{(1)}_{B + i \omega}  = \Coh(X))$ is a very weak stability condition for any variety and a Bridgeland stability condition for curves. By \cite{BriK3, AB}, $ (Z^{(2)}_{\omega,B}, \aA^{(2)}_{B + i \omega} )$ is a Bridgeland stability condition for surfaces. 
In \cite{BMT}, the authors proved that the pair $(Z^{(2)}_{B + i \omega}, \aA^{(2)}_{B + i \omega} )$  is again a 
very weak stability condition for 3-folds. Here the stability was called tilt slope stability.
The usual Bogomolov-Gieseker inequality for $Z^{(1)}_{B + i \omega}$ stable sheaves plays a crucial role in these proofs. Clearly the same arguments work for any higher dimensional varieties. Therefore, we  can always construct the category  $\aA^{(3)}_{B + i \omega}$ when $\dim X \ge 2$. In \cite{BMT}, the authors conjectured that this category is a heart of a Bridgeland stability condition with the central charge $Z_{B+i\omega}$. Moreover, they reduced it to prove Bogomolov-Gieseker type inequalities for $Z^{(2)}_{B + i \omega}$ stable objects $E \in \aA^{(2)}_{B + i \omega}$ with $\Ree Z^{(2)}_{B + i \omega}(E) =0$: 
\begin{equation}
\label{BGineq}
\ch^B_3 (E) \le \frac{\omega^2}{18} \ch_1^B(E).
\end{equation}
So far this conjectural inequality
is known to hold for some Fano 3-folds \cite{BMT, Mac, Sch, Li}, abelian 3-folds \cite{MP1, MP2, Piy, BMS}  and the \'etale quotients of  abelian 3-folds \cite{BMS}. 
\end{rmk}

%%%%%%%%%%%%%%%%%%%%%%%%%%%%%%%%%%
%%%%%%%%%%%%%%%%%%%%%%%%%%%%%%%%%%
\subsection{Fourier-Mukai theory }
%%%%%%%%%%%%%%%%%%%%%%%%%%%%%%%%%%
%%%%%%%%%%%%%%%%%%%%%%%%%%%%%%%%%%
Let us quickly recall some of the important notions in Fourier-Mukai theory.
Further details can be found in \cite{Huy2}.

Let $X,Y$ be smooth projective varieties and let $p_i$, $i=1,2$ be the  projection
maps from $X \times Y$ to $X$ and $Y$, respectively.
The {\it Fourier-Mukai functor} $\Phi_{\eE}^{\SX \to \SY}:
D^b(X) \to D^b(Y)$ with kernel $\eE \in D^b(X \times Y)$
is defined by
$$
\Phi_{\eE}^{\SX \to \SY}(-) = \textbf{R}p_{2*} (\eE \stackrel{\textbf{L}}{\otimes} p_1^*(-)).
$$

The map $\Sigma: Y \times X \to X \times Y$ is defined by $\Sigma(y,x) = (x,y)$ for any $x \in X$ and $y \in Y$.
Let
$ \eE_L = \Sigma^* (\eE^\vee \stackrel{\BL}{\otimes} p_2^*\omega_Y )\, [\dim Y] $, and 
$\eE_R = \Sigma^*(\eE^\vee \stackrel{\BL}{\otimes} p_1^* \omega_X )\, [\dim X]$.
Then we have the adjunctions (see \cite[Proposition 5.9]{Huy2}):
$\Phi_{\eE_L}^{\SY \to \SX} \dashv \Phi_{\eE}^{\SX \to \SY}  \dashv \Phi_{\eE_R}^{\SY \to \SX}$.

When $\Phi_{\eE}^{\SX \to \SY}$ is an equivalence of the derived categories,
usually it is called a {\it Fourier-Mukai transform}.
On the other hand by Orlov's Representability Theorem (see \cite[Theorem 5.14]{Huy2}),
any equivalence between $D^b(X)$ and $D^b(Y)$ is isomorphic
to a Fourier-Mukai  transform $\Phi_{\eE}^{\SX \to \SY}$ for some $\eE \in D^b(X \times Y)$.

Any Fourier-Mukai  functor $\Phi_{\eE}^{\SX \to \SY}: D^b(X) \to D^b(Y)$ induces a linear map
$\Phi^{\SH}_{\eE} : H^{2*}_{\alg}(X, \QQ)  \to   H^{2*}_{\alg}(Y, \QQ)$, usually 
called the cohomological Fourier-Mukai  functor,
 and it is an isomorphism when $\Phi_{\eE}^{\SX \to \SY}$ is a  Fourier-Mukai  transform.
The induced transform fits into the following commutative diagram, due to the Grothendieck-Riemann-Roch theorem.
$$
\xymatrixcolsep{4.5pc}
\xymatrixrowsep{2.25pc}
\xymatrix{
 D^b(X)  \ar[d]_{[-]} \ar[r]^{\Phi_{\eE}^{\SX \to \SY}}  &   D^b(Y) \ar[d]^{[-]} \\
 K(X)  \ar[d]_{v_X(-)}  \ar[r]^{\Phi^K_{\eE}}  &   K(Y) \ar[d]^{v_Y(-)} \\
 H^{2*}_{\alg}(X, \QQ) \ar[r]^{\Phi^{\SH}_{\eE}}  &   H^{2*}_{\alg}(Y, \QQ)
}
$$
Here $v_Z(-) = \ch(-) \sqrt{\text{td}_Z}$ is the Mukai vector map, where
$\ch: K(Z) \to H^{2*}_{\alg}(Z, \QQ)$ is the  Chern character map and $\text{td}_Z$ is the Todd class of $Z$.

Let  $v \in H^{2*}_{\alg}(X, \QQ)$  be a Mukai vector. Then $v= \sum_{i=0}^{\dim X} v_i$ for $v_i \in H^{2i}_{\alg}(X, \QQ)$ and the Mukai dual of
$v$ is  defined by $v^* = \sum_{i=0}^{\dim X} (-1)^i v_i$.
A symmetric bilinear form $\langle - , - \rangle_{\SX}$ called \textit{Mukai pairing} is defined by the formula
\begin{equation*}
\langle v, w \rangle_{\SX} = - \int_X v^* \cdot w \cdot e^{{c_1(X)}/{2} }.
\end{equation*}

Note that for an abelian variety $X$, $\text{td}_X =1$ and $c_1(X) =0$. Hence the Mukai vector $v(E)$ of $E \in D^b(X)$ is the same as its Chern character $\ch(E)$.

Due to Mukai and C\u{a}ld\u{a}raru-Willerton, 
for any $u \in H^{2*}_{\alg}(Y, \QQ)$ and $v \in H^{2*}_{\alg}(X, \QQ)$ we have  
\begin{equation}
\label{isometry}
\left\langle \Phi^{\SH}_{\eE_L}(u) \, , \, v \right\rangle_{\scriptscriptstyle X}
= \left\langle  u \, , \, \Phi^{\SH}_{\eE }(v) \right\rangle_{\scriptscriptstyle Y}
\end{equation}
(see \cite[Proposition 5.44]{Huy2}, \cite{CW}).
%%%%%%%%%%%%%%%%%%%%%%%%%%%%%%%%%%
%%%%%%%%%%%%%%%%%%%%%%%%%%%%%%%%%%
\subsection{Abelian varieties}
%%%%%%%%%%%%%%%%%%%%%%%%%%%%%%%%%%
%%%%%%%%%%%%%%%%%%%%%%%%%%%%%%%%%%
Over any field, an {\it abelian variety} $X$  is a complete group variety, that is
 $X$ is an algebraic variety equipped with the maps
 $m : X \times X \to X, \ (x,y) \mapsto x+y$  (the group law), and
$(-1) : X \to X, \ x \mapsto -x$ (the inverse map),
together with the identity element $e \in X$.
For $a \in X$, the morphism $t_a : X \to X$ is defined by
$t_a = m(-, a) : x \mapsto x + a$.
Over the field of complex numbers, an abelian variety is a complex torus with the structure of a projective algebraic variety.

Let $\Pic^0(X)$ be  the subgroup of the abelian group $\Pic(X)$ consisting of elements represented by line bundles which are algebraically equivalent to zero, and the corresponding quotient $\Pic(X)/ \Pic^0(X)$ is the N\'eron-Severi group 
$\NS(X) $. 
The group $\Pic^0(X)$ is naturally isomorphic to an abelian variety called the \textit{dual abelian variety} of $X$,  denoted by $\HX$.

The \textit{Poincar\'e line bundle} $\pP$ on the product $X \times \HX$ is the
uniquely determined line bundle satisfying
(i) $\pP_{X \times \{\widehat{x}\}} \in \Pic(X)$ is represented by $\widehat{x} \in \HX$, and
(ii) $\pP_{\{e\} \times \HX } \cong \oO_{\HX}$.
In \cite{Muk2}, Mukai proved that the Fourier-Mukai functor $\Phi^{\SX \to \SHX}_{\pP}: D^b(X) \to D^b(\HX)$ is an equivalence of the derived categories, that is a Fourier-Mukai transform.

A vector bundle $E$ on an abelian variety $X$ is called \textit{homogeneous} if we have
$t_x^*E \cong E$ for all $x \in X$. A vector bundle $E$ on $X$ is homogeneous if and only if $E$
can be filtered by line bundles from $\Pic^0(X)$ (see \cite{Muk1}).
We call a vector bundle $E$ is \textit{semihomogeneous} if for every $x \in X$ there exists a flat line bundle $\pP_{X \times \{\widehat{x}\}}$ on $X$
such that $t_x^*E \cong E \otimes  \pP_{X \times \{\widehat{x}\}}$.
A vector bundle $E$ is called \textit{simple} if we have $\End_{\SX}(E) \cong \CC$.
A rank $r$ simple semihomogeneous bundle $E$ has the Chern character 
\begin{equation}
\label{semihomochern}
\ch(E) = r \ e^{c_1(E)/r}.
\end{equation}
Furthermore, due to Mukai, for any $D_{\SX} \in \NS_{\QQ}(X)$, there  exists simple semihomogeneous bundles $E$ on $X$ with 
$\ch(E) = r \, e^{D_{\SX}}$ for some $r \in \ZZ_{>0}$. See \cite{Orl} for further details.

The image of an ample line bundle  $L$ on $X$ under the Fourier-Mukai transform $\Phi_{\pP}^{\SX \to \SHX}$ is
$$
\Phi_{\pP}^{\SX \to \SHX} (L) \cong \HL
$$
for some rank $\chi(L) = c_1(L)^g/g!$ semihomogeneous bundle $\HL$. Moreover, $-c_1(\HL)$ is an ample divisor class on $\HX$ (see \cite{BL1}). Therefore, we have the following:

\begin{lem}
\label{classicalcohomoFMT}
Let $\lx \in \NS_{\QQ}(X)$ be an ample class on $X$. Under the induced cohomological transform $\Phi_{\pP}^{\SH}$ of  $\Phi_{\pP}^{\SX \to \SHX}$ we have
\begin{equation*}
\Phi_{\pP}^{\SH}(e^{\lx}) = ({\lx^g}/{g!}) \, e^{-\lhx}
\end{equation*}
 for some ample class $\lhx \in \NS_{\QQ}(\HX)$, satisfying 
\begin{equation*}
({\lx^g}/{g!}) ({\lhx^g}/{g!}) =1.
\end{equation*}
Moreover, for each $0 \le i \le g$,
\begin{equation*}
\Phi_\pP^{\SH}\left( \frac{ \ell_{\SX}^i}{i!} \right) = \frac{(-1)^{g-i}  \ell_{\SX}^g}{g! (g-i)!} \, \ell_{\SHX}^{g-i}. 
\end{equation*}
\end{lem}

%%%%%%%%%%%%%%%%%%%%%%%%%%%%%%%%%%
%%%%%%%%%%%%%%%%%%%%%%%%%%%%%%%%%%
\section{Cohomological Fourier-Mukai Transforms}
\label{sec:FMTsetup}
%%%%%%%%%%%%%%%%%%%%%%%%%%%%%%%%%%
%%%%%%%%%%%%%%%%%%%%%%%%%%%%%%%%%%
Let $Y$ be a $g$-dimensional abelian variety and let $D_{\SY} \in \NS_{\QQ}(Y)$. Let $X$ be the fine moduli space of rank $r$ simple semihomogeneous bundles $E$ on $Y$ with $c_1(E)/r =D_{\SY}$. Due to Mukai $X$ is a $g$-dimensional abelian variety. Let $\eE$ be the associated universal bundle on $X \times Y$; so by \eqref{semihomochern} we have 
$$
\ch(\eE_{\{x\} \times Y}) = r \, e^{D_{\SY}}.
$$
 Let $\Phi_\eE^{\SX  \to \SY} : D^b(X) \to D^b(Y)$ be the corresponding Fourier-Mukai transform from $D^b(X)$ to $D^b(Y)$ with kernel $\eE$. See \cite{Orl} for further details.
 Then its quasi inverse is given by 
$\Phi_{\Sigma^*\eE^\vee}^{\SY \to \SX}[g]$. Again, by \eqref{semihomochern} we have
$$
\ch(\eE_{X \times \{y\}} ) = r \, e^{D_{\SX}}
$$
for some $D_{\SX} \in \NS_{\QQ}(X)$.

A \textit{polarization} on $X$ is by definition the first Chern class $c_1(L)$ of an ample line bundle $L$ on $X$. 
However, it is usual to say the line bundle $L$ itself a polarization.

Let $a \in X$ and $b \in Y$.
Consider the Fourier-Mukai functor $\Gamma$ from $D^b(X)$ to $D^b(\HY)$ defined by 
$$
\Gamma = \Phi_\pP^{\SY  \to \SHY} \circ \eE_{\{a\}\times Y}^* \circ \Phi_\eE^{\SX  \to \SY} \circ \eE_{X \times \{b\}}^*\, [g],
$$
where $\eE_{\{a\}\times Y}^*$ denotes the functor $\eE_{\{a\}\times Y} \otimes^*(-)$ and similar for $ \eE_{X \times \{b\}}^*$. 
Let $\widehat \Gamma: D^b(\HY) \to D^b(X) $ be the  Fourier-Mukai functor defined by 
$$
\widehat \Gamma =    \eE_{X \times \{b\}} \circ \Phi_{\Sigma^*\eE^\vee}^{\SY\to \SX  }  \circ \eE_{\{a\}\times Y} \circ\Phi_{\Sigma^*\pP^\vee}^{ \SHY \to \SY }\, [g].
$$
Then $\widehat \Gamma $ and $\Gamma$ are adjoint functors to each other. By direct computation, 
$\Gamma (\oO_{x}) = \oO_{Z_x}$ for some $0$-subscheme $Z_x \subset \HY$, and 
$\Gamma (\oO_{\widehat y}) = \oO_{Z_{\widehat y}}$ for some $0$-subscheme $Z_{\widehat y} \subset  X$;
where the lengths of $Z_x$ and $Z_{\widehat y}$ are $r^3$ and $r$ respectively. 
Therefore, the Fourier-Mukai kernel $\fF$ of $\Gamma$ is $\fF \in  \Coh_g(X \times \HY)$, with 
$\fF^{\vee} \cong \calExt^g(\fF, O_{X \times \HY})[-g]$.
So $\Gamma (\Coh_{i}(X)) \subset \Coh_i(\HY)$ and $\widehat \Gamma (\Coh_i (\HY)) \subset \Coh_i(X)$ for all $i$. 
Also by direct computation,  $\Gamma(\oO_X)$ and $\widehat \Gamma (\oO_{\HY}) $ are homogeneous bundles of rank $r$ and $r^3$ respectively. 
Let $\lx \in \NS_{\QQ}(X)$ be an ample class.
Under the induced cohomological map $ \Gamma^{\SH}$ we have 
$$
\Gamma^{\SH}(e^{\ell_{\SX}}) = r\  e^{\ell_{\SHY}}, 
$$
for some class $\ell_{\SHY} \in \NS_{\QQ}(\HY)$ satisfying $r^2 \, {\lx^g} =  {\lhy^g}$. 
So we have 
$$
\widehat \Gamma^{\SH}(e^{\ell_{\SHY}}) = r^3 \  e^{\ell_{\SX}}.
$$
Moreover, for each $0 \le i \le g$,
$$
\Gamma^{\SH}( \ell_{\SX}^i) = r  \, \ell_{\SHY}^i, \  \ \
\widehat \Gamma^{\SH}(\lhy^i) = r^3 \,  \lx^i.
$$

% For $0 \le j \le g$, let $\HY^{j} \subset \HY$ be a closed $j$-dimensional subscheme of $\HY$.  
%Then we have
%\begin{align*}
%\int_{\HY} \ell_{\SHY}^{g-j} \cdot [\HY^{j}] 
%& =  \frac{1}{r^4}  \int_{\HY} \ell_{\SHY}^{g-j} \cdot \Gamma^{\SH} \widehat \Gamma^{\SH} [\HY^{j}] \\
%& =  \frac{(-1)^{g-j}}{r^4} \left\langle  \ell_{\SHY}^{g-j} , \,  \Gamma^{\SH} \widehat \Gamma^{\SH} [\HY^{j}]  \right\rangle_{\SHY} \\
%& =  \frac{(-1)^{g-j}}{r^4} \left\langle  \widehat \Gamma^{\SH}(\ell_{\SHY}^{g-j}) , \,   \widehat \Gamma^{\SH} [\HY^{j}]  \right\rangle_{\SX} \ \ \ \text{by \eqref{isometry}} \\
%& =  \frac{(-1)^{g-j}}{r} \left\langle \ell_{\SX}^{g-j}  , \,   \widehat \Gamma^{\SH} [\HY^{j}]  \right\rangle_{\SX} \\
%& = \frac{1}{r}  \int_{X} \ell_{\SX}^{g-j}  \cdot \widehat \Gamma^{\SH} [\HY^{j}] > 0,
%\end{align*}
%as $\widehat{\Gamma}^{\SH} [\HY^{j}] $ is a class of $j$-subscheme on $X$ and $\ell_{\SX}$ is an ample class.
%Hence, by the Nakai--Moishezon criterion, $\ell_{\SHY}$ is an ample class on $\HY$.
 
By the Nakai--Moishezon criterion, one can show $\ell_{\SHY}$ is an ample class on $\HY$.
 By Theorem \ref{classicalcohomoFMT}, under the induced cohomological map of $\Phi_{\Sigma^*\pP^\vee}^{ \SHY \to \SY }$ we have
$$
\Phi_{\Sigma^*\pP^\vee}^{\SH}(e^{\ell_{\SHY}}) =({\ell_{\SHY}^g}/{g!})  \, e^{- \ell_{\SY}},
$$
for some ample class $\ell_{\SY} \in \NS_{\QQ}(Y)$.

Let  $\Xi : D^b(X) \to D^b(Y)$ be the Fourier-Mukai functor defined by 
$$
\Xi = \eE_{\{a\}\times Y}^* \circ \Phi_\eE^{\SX  \to \SY} \circ \eE_{X \times \{b\}}^* =\Phi_{\Sigma^*\pP^\vee}^{ \SHY \to \SY } \circ \Gamma.
$$
The image of $e^{\ell_{\SX}}$ under its induced cohomological transform $\Xi^{\SH}$  is 
$ ({r^3 \lx^g}/{g!}) \, e^{- \ell_{\SY}} $. Therefore, we deduce the following.

\begin{thm} [\cite{Piy2}]
\label{generalcohomoFMT}
 If $\lx \in \NS_{\QQ}(X)$ is an ample class then 
$$
 e^{- D_{\SY}} \, \Phi_{\eE}^{\SH} \, e^{-D_{\SX}} ( e^{\ell_{\SX}}) =  (r \, {\lx^g}/{g!}) \,  e^{-\ell_{\SY}},
$$
for some ample class $\ly \in \NS_{\QQ}(Y)$, satisfying $({\lx^g}/g!)({\ly^g}/g!)=  1/r^2$.
Moreover,  for each $0 \le i \le g$,
\begin{equation*}
 e^{- D_{\SY}} \, \Phi_{\eE}^{\SH} \, e^{-D_{\SX}} \left( \frac{ \ell_{\SX}^i}{i!} \right) = \frac{(-1)^{g-i}  r \, \ell_{\SX}^g}{g! (g-i)!} \, \ly^{g-i}. 
\end{equation*}
\end{thm}

\begin{thm}[\cite{Piy2}]
\label{polarize}
If the ample line bundle $L$ defines a polarization on $X$,  then
the ample line bundle $\det (\Xi(L))^{-1}$ defines a polarization on $Y$.
\end{thm}

Let us introduce the following notation:
\begin{nota}
\rm
Let $B, \lx \in \NS_{\QQ}(X)$. For $E \in D^b(X)$, the entries $v^{B,{\lx}}_i(E)$, $i=0, \ldots, g$ are defined by
$$
v^{B,{\lx}}_i(E) = i! \, \lx^{g-i} \cdot \ch^B_{i}(E).
$$
Here $\ch^B_{i}(E)$ is the $i$-th component of the $B$-twisted Chern character $\ch^B(E)$ defined by $\ch^B(E) = e^{-B} \ch(E)$. 
The vector $v^{B,{\lx}}(E)$ is defined by 
\begin{equation*}
v^{B,{\lx}}(E) = \left( v^{B,{\lx}}_0(E) , \ldots, v^{B,{\lx}}_g(E) \right).
\end{equation*}
\end{nota}

\begin{thm}[\cite{Piy2}]
\label{antidiagonalrep}
If we consider $v^{-D_{\SX} ,\lx}, v^{D_{\SY},\ly}$ as column vectors, then 
$$
v^{D_{\SY},\ly}\left(\Phi_\eE^{\SX  \to \SY}(E) \right) = \frac{g!}{r \, \ell_{\SX}^g} \,  \Adiag\left(1,-1,\ldots, (-1)^{g-1}, (-1)^g\right)  \ v^{-D_{\SX}, \lx}(E).
$$
\end{thm}

%
%\begin{proof}
%The $i$-th entry of $v^{D_{\SY},\ly}\left(\Phi_\eE^{\SX  \to \SY}(E) \right) $ is
%\begin{align*}
%v^{D_{\SY},\ly}_{i} \left(\Phi_\eE^{\SX  \to \SY}(E) \right)
%& = i! \, \ly^{g-i} \cdot \ch^{D_{\SY}}_{i}\left(\Phi_\eE^{\SX  \to \SY}(E) \right) \\
%& = i!  \int_{Y} {\ell_{\SY}^{g-i}} \cdot e^{-D_{\SY}}\ch\left(\Phi_\eE^{\SX  \to \SY}(E) \right) \\
%& =(-1)^{g-i}  i!  \left\langle {\ell_{\SY}^{g-i}} , \, e^{-D_{\SY}} \ch\left(\Phi_\eE^{\SX  \to \SY}(E) \right) \right\rangle_{\SY}\\
%& =(-1)^{g-i}  i!  \left\langle {\ell_{\SY}^{g-i}} , \, e^{-D_{\SY}} \Phi_\eE^{\SH}(\ch(E))  \right\rangle_{\SY}\\
%& =(-1)^{g-i}  i!  \left\langle {\ell_{\SY}^{g-i}} , \, e^{-D_{\SY}} \Phi_\eE^{\SH}e^{-D_{\SX}}(\ch^{-D_{\SX}}(E))  \right\rangle_{\SY}\\
%& =(-1)^{g-i}  i!  \left\langle {\left(e^{-D_{\SY}} \Phi_\eE^{\SH}e^{-D_{\SX}}\right)^{-1}(\ell_{\SY}^{g-i}}) , \, \ch^{-D_{\SX}}(E) \right\rangle_{\SX}\\
%& =\frac{ g! (g-i)! }{r \, \lx^g}  \left\langle \lx^i , \, \ch^{-D_{\SX}}(E) \right\rangle_{\SX}\\
%& =\frac{ (-1)^i g! (g-i)! }{r  \, \lx^g}  \int_X  \lx^i \cdot \ch^{-D_{\SX}}(E) \\
%& =\frac{ (-1)^i g! (g-i)! }{r \, \lx^g}  \,  \lx^i \cdot \ch^{-D_{\SX}}_{g-i}(E) \\
%& =\frac{ (-1)^i g! }{r  \, \lx^g}  \, v^{-D_{\SX},\lx}_{g-i} (E). 
%\end{align*}
%Hence, the claim follows as required. 
%\end{proof}
%%%%%%%%%%%%%%%%%%%%%%%%%%%%%%%%%%
%%%%%%%%%%%%%%%%%%%%%%%%%%%%%%%%%%
\section{Stability Conditions Under  Fourier-Mukai Transforms}
%%%%%%%%%%%%%%%%%%%%%%%%%%%%%%%%%%
%%%%%%%%%%%%%%%%%%%%%%%%%%%%%%%%%%
This section generalizes some of the similar results in \cite{MP2, Piy}.
Recall that a Bridgeland stability condition $\sigma$ on a triangulated category
$\dD$ consists of a stability function $Z$ together with a slicing $\pP$
of $\dD$ satisfying certain axioms. Equivalently, one can define
$\sigma$ by giving a bounded t-structure on $\dD$ together with a
stability function $Z$ on the corresponding heart $\aA$ satisfying the
Harder-Narasimhan property. Then $\sigma$ is usually written as the pair $(Z, \pP)$
or $(Z, \aA)$. 

Let $\Upsilon:   \dD \to \dD'$ be an equivalence of triangulated categories, and let $W: K(\dD) \to \CC$ be a group homomorphism. Then
$$
\left( \Upsilon \cdot W \right) ([E]) = W \left( [ \Upsilon^{-1}(E) ] \right)
$$
defines an induced group morphism $ \Upsilon \cdot W$ in  $\Hom (K(\dD'),
\CC)$ by the equivalence $\Upsilon$. Moreover, this can be extended to a natural induced stability condition 
 on $\dD'$ by defining
$\Upsilon \cdot (Z, \aA) = (\Upsilon \cdot Z , \Upsilon(\aA))$.

Let $X, Y$ be two derived equivalent $g$-dimensional abelian varieties as in Section \ref{sec:FMTsetup}, which is given by the Fourier-Mukai transform $\Phi_{\eE}^{\SX \to \SY} : D^b(X) \to D^b(Y)$. 
Let $\lx \in \NS_{\QQ}(X)$ be an ample class on $X$ and let $\ly \in \NS_{\QQ}(Y)$ be the induced ample class on $Y$ as in Theorem \ref{generalcohomoFMT}. 

Let $u = \lambda e^{i \alpha}$ be a complex number. Consider the 
function $Z^{\SX}_{-D_{\SX}+u\lx}: K(X) \to \CC$ defined by 
$$
Z^{\SX}_{-D_{\SX}+u\lx}(E)= -\int_{X} e^{-\left( -D_{\SX}+u\lx \right)}\ch(E).
$$
For $E \in D^b(Y)$ we have
\begin{align*}
\left( \Phi_{\eE}^{\SX \to \SY}  \cdot Z^{\SX}_{-D_{\SX}+u\lx}\right)(E) 
& = \left\langle  e^{ -D_{\SX}+u\lx } , \ \ch\left(\left(\Phi_{\eE}^{\SX \to \SY}\right)^{-1}(E)\right)  \right\rangle_{\SX} \\
& = \left\langle    e^{ -D_{\SX}+ u\lx } , \ \left(\Phi_{\eE}^{\SH}\right)^{-1} \left(\ch(E)\right)  \right\rangle_{\SX} \\
& = \left\langle  \Phi_{\eE}^{\SH} \left( e^{ -D_{\SX}+ u\lx }\right) , \ \ch(E)  \right\rangle_{\SY}\\
& = \left\langle e^{D_{\SY}} \left( e^{-D_{\SY}}  \Phi_{\eE}^{\SH} e^{ -D_{\SX}}\right) (e^{ u\lx })  , \ \ch(E)  \right\rangle_{\SY} \\
& =  (r \, \lx^g u^g /g!) \ \left \langle   e^{D_{\SY} -\ly/u}  , \ \ch(E)  \right\rangle_{\SY},
\end{align*}
 since by Theorem \ref{generalcohomoFMT},   
$e^{-D_{\SY}}  \Phi_{\eE}^{\SH} e^{ -D_{\SX}}(e^{ u\lx }) = (r \, \lx^g u^g /g!) \, e^{-\ly/u} $. So we have the following relation:

\begin{lem}
\label{FMTact}
We have $\Phi_{\eE}^{\SX \to \SY}  \cdot Z^{\SX}_{-D_{\SX}+u\lx} = \zeta\ Z^{\SY}_{D_{\SY} - \ly/u}(E)$, for $\zeta = r \, \lx^g u^g /g!$.
\end{lem}

Assume there exist a stability condition for any complexified ample class  $-D_{\SX}+u\lx$ with a heart $\aA^{\SX}_{-D_{\SX}+u\lx}$ and a slicing $\pP^{\SX}_{-D_{\SX}+u\lx}$ associated to the central charge function $Z^{\SX}_{-D_{\SX}+u\lx}$.  
From Lemma \ref{FMTact} for any $\phi \in \RR$,  \ $\zeta \, Z^{\SY}_{D_{\SY} - \ly/u}\left( \Phi_{\eE}^{\SX \to \SY}\left( \pP^{\SX}_{-D_{\SX}+u\lx}(\phi) \right)\right) \subset \RR_{>0} e^{i\pi \phi}$; that is 
$$
Z^{\SY}_{D_{\SY} - \ly/u}\left( \Phi_{\eE}^{\SX \to \SY}\left(  \pP^{\SX}_{-D_{\SX}+u\lx}(\phi)  \right)\right)  \subset \RR_{>0} e^{i\left(\pi \phi - \arg(\zeta)\right)}.
$$
So  we would expect
$$
 \Phi_{\eE}^{\SX \to \SY}\left( \pP^{\SX}_{-D_{\SX}+u\lx} (\phi) \right) = \pP^{\SY}_{D_{\SY}- \ly/u} \left( \phi - \frac{\arg (\zeta)}{\pi}\right),
$$
and so 
$$
 \Phi_{\eE}^{\SX \to \SY}\left( \pP^{\SX}_{-D_{\SX}+u\lx} ((0,\,1])\right) =  \pP^{\SY}_{D_{\SY}- \ly/u} \left(\left( -\frac{\arg (\zeta)}{\pi},  \, -\frac{\arg (\zeta)}{\pi} +1\right]\right).
$$
For $0\le \alpha <1$, 
$  \pP^{\SY}_{D_{\SY}- \ly/u} ((\alpha, \alpha+1]) =\left \langle   \pP^{\SY}_{D_{\SY}- \ly/u} ((0, \alpha])\,[1] , \,   \pP^{\SY}_{D_{\SY}- \ly/u}  ((\alpha, 1]) \right \rangle$ is a tilt of $  \aA^{\SY}_{D_{\SY}- \ly/u}  =   \pP^{\SY}_{D_{\SY}- \ly/u} ((0,1])$
 associated to a torsion theory coming from $Z^{\SY}_{D_{\SY}- \ly/u} $ stability. 
Therefore, one would expect $  \Phi_{\eE}^{\SX \to \SY} \left( \aA^{\SX}_{-D_{\SX} + u\lx } \right)$  is a
tilt of $\aA^{\SY}_{D_{\SY}- \ly/u}$ associated to a torsion theory coming from $Z^{\SY}_{D_{\SY}- \ly/u}$ stability, up to shift. 

Moreover, for the Fourier-Mukai transform $\Phi_{\eE}^{\SX \to \SY}$ when $\zeta$ is real, that is, 
$u^g \in \RR$, 
we would expect the following equivalence:

\begin{conj}
\label{conjequivalence}
The Fourier-Mukai transform $\Phi_{\eE}^{\SX \to \SY}: D^b(X) \to D^b(Y)$ gives the equivalence of stability condition hearts conjecturally constructed in Conjecture \ref{conjstab}:
$$
\Phi_{\eE}^{\SX \to \SY} [k]  \left(  \aA^{\SX}_{\Omega_k}  \right) = \aA^{\SY}_{\Omega'_k}.
$$
Here  $\Omega_k = -D_{\SX} + \lambda e^{i k\pi/g }\, \lx $ and $\Omega'_k = D_{\SY}  - e^{-i k \pi/g} \, \ly/\lambda$
are complexified ample classes on $X$ and $Y$ respectively, for   $k \in \{1, 2, \ldots, (g-1)\}$, and some $ \lambda \in \RR_{>0}$. 
\end{conj}

%%%%%%%%%%%%%%%%%%%%%%%%%%%%%%%%%%
%%%%%%%%%%%%%%%%%%%%%%%%%%%%%%%%%%
\section{Equivalences of Stability Condition Hearts}
%%%%%%%%%%%%%%%%%%%%%%%%%%%%%%%%%%
%%%%%%%%%%%%%%%%%%%%%%%%%%%%%%%%%%
The main aim of this section is to show that Conjecture \ref{conjequivalence} is true for abelian surfaces and abelian 3-folds. In order to achieve this we study various stabilities of sheaves and complexes of them under the Fourier-Mukai transforms. 
\subsection{Abelian surfaces}
%%%%%%%%%%%%%%%%%%%%%%%%%%%%%%%%%%
%%%%%%%%%%%%%%%%%%%%%%%%%%%%%%%%%%
Let $X, Y$ be derived equivalent abelian surfaces  and let $\lx, \ly$ be ample classes on them respectively as in Theorem \ref{generalcohomoFMT}.
Let $\Psi$ be the Fourier-Mukai transform $\Phi_{\eE}^{\SX \to \SY}$ from $X$ to $Y$ with kernel  $\eE$, and let $\HPsi =  \Phi_{\Sigma^* \eE^\vee}^{\SY \to \SX}$. We have 
\begin{equation}
\label{imagesurface}
\Psi(\Coh(X)) \subset \langle \Coh(Y), \Coh(Y)[-1], \Coh(Y)[-2] \rangle,
\end{equation}
and similar relation for $\HPsi$.
Since 
 $\HPsi \circ \Psi \cong  [-2]$ and $\Psi \circ \HPsi \cong  [-2]$, 
 we have the following convergence of the spectral sequences.

\begin{equation}
\label{mukaispecseq}
\left.\begin{aligned}
 &  E_2^{p,q} = \HPsi^{p} \Psi^q(E) \Longrightarrow H^{p+q-2}_{\Coh(X)}(E) \\        
 &  E_2^{p,q} = \Psi^{p}\HPsi^q(E) \Longrightarrow H^{p+q-2}_{\Coh(Y)}(E)
       \end{aligned}
  \ \right\}.
\end{equation}
Here and elsewhere we write $\HPsi^{p}(E) = H^p_{\Coh(X)} (\HPsi(E))$ and $\Psi^{q}(E) = H^q_{\Coh(Y)} (\Psi(E))$. Immediately from the convergence of this spectral sequence for $E \in \Coh(X)$, we deduce that
\begin{itemize}[leftmargin=1cm]
\item $\HPsi \Psi^0(E) \in \Coh(X)[-2] $, and $\HPsi \Psi^2(E) \in  \Coh(X) $;
\item there is an injection $\HPsi^0\Psi^1(E) \hookrightarrow \HPsi^2\Psi^0(E)$, and
 a surjection $ \HPsi^0\Psi^2(E) \twoheadrightarrow \HPsi^2\Psi^1(E)$.
\end{itemize}

Let us recall the notation in Conjecture \ref{conjequivalence} for our derived equivalent abelian surfaces. 
Consider  the complexified ample classes $\Omega = -D_{\SX} + i \lambda \lx$, 
$\Omega' = D_{\SY} + i   \ly/\lambda $ on $X$, $Y$ respectively. 
The function defined by $Z^{(1)}_{\Omega} = -i \int_{X} e^{- \Omega } (\ch_0, \ch_1, 0)$ together with the 
standard heart $\Coh(X)$ defines a very weak stability condition $\sigma_1$ on $D^b(X)$. 
Define the subcategories 
$$
\fF^{\SX} = \pP^{\SX}_{\sigma_1}((0,\, 1/2]), \ \ \tT^{\SX}= \pP^{\SX}_{\sigma_1}((1/2 ,\, 1])
$$
of $\Coh(X)$ in terms of the associated slicing $\pP^{\SX}_{\sigma_1}$.
Then the Bridgeland  stability condition heart in Conjecture \ref{conjequivalence} is 
$$
\aA^{\SX} = \langle \fF^{\SX}[1] , \tT^{\SX} \rangle = \pP^{\SX}_{\sigma_1}((1/2 ,\, 3/2]).
$$
We consider similar subcategories  associated to $\Omega'$ on $Y$. 

Let us sketch the proof of the equivalence in Conjecture \ref{conjequivalence}. This is essentially a generalization of Yoshioka's in \cite{Yos}, and closely follow the one in the author's PhD thesis \cite[Section 6]{Piy}.

We need the following results about cohomology sheaves of the images under the Fourier-Mukai transforms.

\begin{prop}[\cite{Piy2}]
Let $E \in \Coh(X)$. Then we have the following:
\begin{enumerate}[label=\arabic*.]
\item (i) If $E \in \tT^{\SX}$ then $\Psi^2(E) = 0$, and (ii) if $E \in  \fF^{\SX}$ then $\Psi^0(E) = 0$.
\item (i) $\Psi^2(E) \in \tT^{\SY}$, and (ii) $\Psi^0(E) \in \fF^{\SY}$.
\item (i) if $E \in \tT^{\SX}$ then $\Psi^1(E) \in \tT^{\SY}$, and (ii) if $E \in \fF^{\SX}$ then $\Psi^1(E) \in \fF^{\SY}$.
\end{enumerate}
\end{prop}

In other words, the results of the above proposition say
\begin{equation*}
\left.\begin{aligned}
 & \Psi (\tT^{\SX}) \subset \langle \fF^{\SY} , \tT^{\SY}[-1]  \rangle \\        
 &  \Psi (\fF^{\SX}) \subset \langle \fF^{\SY}[-1] , \tT^{\SY}[-2]  \rangle
       \end{aligned}
  \ \right\}.
\end{equation*}
Similar results hold for $\HPsi$. Since $\aA^{\SX} = \langle \fF^{\SX}[1] , \tT^{\SX} \rangle$ and  $\aA^{\SY} = \langle \fF^{\SY}[1] , \tT^{\SY} \rangle$,
we have $\Psi [1] (\aA^{\SX} ) = \aA^{\SY}$
as required.

%%%%%%%%%%%%%%%%%%%%%%%%%%%%%%%%%%
%%%%%%%%%%%%%%%%%%%%%%%%%%%%%%%%%%
\subsection{Abelian 3-folds}
%%%%%%%%%%%%%%%%%%%%%%%%%%%%%%%%%%
%%%%%%%%%%%%%%%%%%%%%%%%%%%%%%%%%%
Let $X, Y$ be derived equivalent abelian 3-folds  and let $\lx, \ly$ be ample classes on them respectively as in Theorem \ref{generalcohomoFMT}.
Let $\Psi$ be the Fourier-Mukai transform $\Phi_{\eE}^{\SX \to \SY}$ from $X$ to $Y$ with kernel  $\eE$, and let $\HPsi =  \Phi_{\Sigma^* \eE^\vee}^{\SY \to \SX}$. Then 
 $\HPsi \circ \Psi \cong  [-3] $ and $\Psi \circ \HPsi \cong  [-3] $.

Let us recall the notation in Conjecture \ref{conjequivalence} for our derived equivalent abelian 3-folds. 
Consider the complexified ample classes 
\begin{equation*}
\left.\begin{aligned}
 &\Omega = \left( -D_{\SX} +  \lambda \lx/2 \right)  + i \sqrt{3} \lambda \lx/2 \\        
 &  \Omega' = \left( D_{\SY} - \ly/(2 \lambda) \right) + i  \sqrt{3}  \ly/(2 \lambda)
       \end{aligned}
  \ \right\}
\end{equation*}
on $X$, $Y$ respectively. 
The function defined by $Z^{(1)}_{\Omega} =  \int_{X} e^{- \Omega } (\ch_0, \ch_1, 0,0)$ together with the 
standard heart $\Coh(X)$ defines a very weak stability condition $\sigma_1$ on $D^b(X)$. 
Define the subcategories 
$$
\fF^{\SX}_1= \pP^{\SX}_{\sigma_1}((0,\, 1/2]), \ \ \tT^{\SX}_1= \pP^{\SX}_{\sigma_1}((1/2 ,\, 1])
$$
of $\Coh(X)$ in terms of the associated slicing $\pP^{\SX}_{\sigma_1}$; and the corresponding tilt be
$$
\bB^{\SX} = \langle \fF^{\SX}_1[1] , \tT^{\SX}_1 \rangle = \pP^{\SX}_{\sigma_1}((1/2 ,\, 3/2]).
$$
Then as mentioned in Remark \ref{remarkconjstab}, due to \cite{BMT} the function defined by 
$$Z^{(2)}_{\Omega} = -i \int_{X} e^{- \Omega } (\ch_0, \ch_1, \ch_2, 0)$$
 together with the  heart $\bB^{\SX}$ defines a very weak stability condition $\sigma_2$ on $D^b(X)$. Define the subcategories 
$$
\fF^{\SX}_2 = \pP^{\SX}_{\sigma_2}((0,\, 1/2]), \ \ \tT^{\SX}_2= \pP^{\SX}_{\sigma_2}((1/2 ,\, 1])
$$
of $\bB^{\SX}$ in terms of the associated slicing $\pP^{\SX}_{\sigma_2}$.
Then the expected Bridgeland  stability condition heart in Conjecture \ref{conjstab} is 
$$
\aA^{\SX} = \langle \fF^{\SX}_2[1] , \tT^{\SX}_2 \rangle = \pP^{\SX}_{\sigma_2}((1/2 ,\, 3/2]).
$$
We consider similar subcategories on $Y$ associated to $\Omega'$. 

Under the Fourier-Mukai transform $\Psi$, for any $E \in \Coh(X)$, 
$H^{i}_{\Coh(Y)}(\Psi (E)) = 0$ when $i \not \in \{0,1,2,3\}$. 

We can visualize $\bB^{\SX}$ and $\bB^{\SY}$ as follows:
$$
\begin{tikzpicture}[scale=1.2]
\draw[style=dashed] (0.5,0) grid (3.5,1);
\fill[lightgray] (1,0) -- (2,0)  to[out=25,in=-115] (3,1) -- (2,1) to[out=-115,in=25] (1,0);
\draw[style=thick] (1,0) -- (2,0) to[out=25,in=-115] (3,1)-- (2,1) to[out=-115,in=25] (1,0);
\draw[style=thick] (2,0) -- (2,1);
\draw (1.75,0.25) node {$\scriptscriptstyle B$};
\draw (2.3,0.7) node {$\scriptscriptstyle A$};
\draw (1.5,- 0.3) node {$\scriptstyle{-1}$};
\draw (2.5,-0.3) node {$\scriptstyle{0}$};
\draw (-1.2,0.5) node {$\bB^{\SX} = \langle \mathcal{F}^{\SX}_1[1] , \mathcal{T}^{\SX}_1 \rangle : $};
\draw (5,0.5) node {$\scriptstyle A \in \mathcal{T}^{\SX}_1,  \ \ B \in \mathcal{F}^{\SX}_1$ };
\end{tikzpicture}
$$
$$
\begin{tikzpicture}[scale=1.2]
\draw[style=dashed] (0.5,0) grid (3.5,1);
\fill[lightgray] (1,0) -- (2,0) to[out=65,in=-155] (3,1) -- (2,1) to[out=-155,in=65] (1,0);
\draw[style=thick] (1,0) -- (2,0) to[out=65,in=-155] (3,1)-- (2,1) to[out=-155,in=65] (1,0);
\draw[style=thick] (2,0) -- (2,1);
\draw (1.6,0.4) node {$\scriptscriptstyle D$};
\draw (2.25,0.75) node {$\scriptscriptstyle C$};
\draw (1.5,- 0.3) node {$\scriptstyle{-1}$};
\draw (2.5,-0.3) node {$\scriptstyle{0}$};
\draw (-1.2,0.5) node {$\bB^{\SY} = \langle \mathcal{F}^{\SY}_1[1] , \mathcal{T}^{\SY}_1\rangle : $};
\draw (5,0.5) node {$\scriptstyle C \in \mathcal{T}^{\SY}_1, \ \ D \in \mathcal{F}^{\SY}_1$ };
\end{tikzpicture}
$$

Following similar arguments in \cite{MP1, MP2, Piy}, one can prove 
\begin{equation*}
\left.\begin{aligned}
 & H^0_{\Coh(Y)} ( \Psi(\fF^{\SX}_1)) \subset \fF^{\SY}_1, \ \ H^3_{\Coh(Y)} (\Psi(\fF^{\SX}_1)) \subset \tT^{\SY}_1\\        
 &  H^3_{\Coh(Y)} (\Psi(\tT^{\SX}_1)) = \{0\}, \ \ H^2_{\Coh(Y)} (\Psi(\tT^{\SX}_1)) \subset \tT^{\SY}_1
       \end{aligned}
  \ \right\}.
\end{equation*}
That is
$\Psi\left( \fF^{\SX}_1 \right)[1] \subset \langle \bB^{\SY} ,\bB^{\SY}[-1], \bB^{\SY}[-2] \rangle$,
$$
\begin{tikzpicture}[scale=1.2]
\draw[style=dashed] (0.5,-2) grid (6.5,-1);
\fill[lightgray] (1,-2) -- (4,-2) to[out=65,in=-155] (5,-1) -- (2,-1) to[out=-155,in=65] (1,-2);
\draw[style=thick] (1,-2) -- (4,-2) to[out=65,in=-155] (5,-1) -- (2,-1) to[out=-155,in=65] (1,-2);
\draw[style=dashed] (2,-2) -- (2,-1) ;
\draw[style=dashed] (3,-2) -- (3,-1) ;
\draw[style=dashed] (4,-2) -- (4,-1) ;
\draw (1.6,-1.7) node {$\scriptscriptstyle B_{\Psi}^0$};
\draw (2.5,-1.5) node {$\scriptscriptstyle B_{\Psi}^1$};
\draw (3.5,-1.5) node {$\scriptscriptstyle B_{\Psi}^2$};
\draw (3.5,-2.8) node {$\scriptstyle B_{\Psi}^i = H^i_{\Coh(Y)}(\Psi(B))$};
\draw (4.25,-1.25) node {$\scriptscriptstyle B_{\Psi}^3$};
\draw (1.5,- 2.3) node {$\scriptscriptstyle{-1}$};
\draw (2.5,-2.3) node {$\scriptscriptstyle{0}$};
\draw (3.5,- 2.3) node {$\scriptscriptstyle{1}$};
\draw (4.5,-2.3) node {$\scriptscriptstyle{2}$};
\draw (5.5,-2.3) node {$\scriptscriptstyle{3}$};
%%%%%
\draw[style=dashed] (-3 ,-2) grid (-1,-1);
\fill[lightgray] (-3,-2) -- (-2,-2) -- (-2,-1) to[out=-115,in=25] (-3,-2);
\draw[style=thick] (-3,-2) -- (-2, -2) to[out=25,in=-115] (-1, -1) -- (-2, -1) to[out=-115,in=25] (-3, -2);
\draw[style=thick] (-2,-2) -- (-2, -1);
\draw (-2.25,-1.75) node {$\scriptscriptstyle B$};
\draw (-2.5,- 2.3) node {$\scriptscriptstyle{-1}$};
\draw (-1.5,-2.3) node {$\scriptscriptstyle{0}$};
\draw (-3.2,-2.1) to [out=120,in=240] (-3.2,-0.9);
\draw (-0.8,-2.1) to [out=60,in=300] (-0.8,-0.9);
\draw (-3.7,-1.5) node {$\Psi$};
\draw (0,-1.5) node {$=$};
\end{tikzpicture}
$$
and $\Psi \left( \tT^{\SX}_1 \right) \subset \langle \bB^{\SY}, \bB^{\SY}[-1], \bB^{\SY}[-2] \rangle$.
$$
\begin{tikzpicture}[scale=1.2]
%%%%%
\draw[style=dashed] (0.5,0) grid (6.5,1);
\fill[lightgray] (2,0) -- (4,0)to[out=65,in=-155] (5,1) -- (2,1) --(2,0);
\draw[style=thick] (2,0) -- (4,0)  to[out=65,in=-155]  (5,1) -- (2,1) -- (2,0);
\draw[style=dashed] (3,0) -- (3,1) ;
\draw[style=dashed] (4,0) -- (4,1) ;
\draw (2.5,0.5) node {$\scriptscriptstyle A_{\Psi}^0$};
\draw (3.5,0.5) node {$\scriptscriptstyle A_{\Psi}^1$};
\draw (3.5,-0.8) node {$\scriptstyle A_{\Psi}^i = H^i_{\Coh(Y)}(\Psi(A))$};
\draw (4.25,0.75) node {$\scriptscriptstyle A_{\Psi}^2$};
\draw (1.5,- 0.3) node {$\scriptscriptstyle{-1}$};
\draw (2.5,-0.3) node {$\scriptscriptstyle{0}$};
\draw (3.5,- 0.3) node {$\scriptscriptstyle{1}$};
\draw (4.5,-0.3) node {$\scriptscriptstyle{2}$};
\draw (5.5,-0.3) node {$\scriptscriptstyle{3}$};
%%%%%
\draw[style=dashed] (-3 ,0) grid (-1,1);
\fill[lightgray] (-2,0) to[out=25,in=-115] (-1,1) -- (-2,1) -- (-2,0);
\draw[style=thick] (-3,0) -- (-2, 0) to[out=25,in=-115]  (-1, 1) -- (-2, 1) to[out=-115,in=25]  (-3, 0);
\draw[style=thick] (-2,0) -- (-2, 1);
\draw (-1.7,0.7) node {$\scriptscriptstyle A$};
\draw (-2.5,- 0.3) node {$\scriptscriptstyle{-1}$};
\draw (-1.5,-0.3) node {$\scriptscriptstyle{0}$};
\draw (-3.2,-0.1) to [out=120,in=240] (-3.2,1.1);
\draw (-0.8,-0.1) to [out=60,in=300] (-0.8,1.1);
\draw (-3.7,0.5) node {$\Psi$};
\draw (0,0.5) node {$=$};
\end{tikzpicture}
$$
Hence, 
$$
\Psi \left( \bB^{\SX} \right) \subset \langle \bB^{\SY}, \bB^{\SY}[-1], \bB^{\SY}[-2] \rangle.
$$ 
Similarly, one can prove 
$$
\Psi[1] (\bB^{\SY}) \subset \langle \bB^{\SX} , \bB^{\SX}[-1], \bB^{\SX}[-2]\rangle.
$$
Therefore, under the Fourier-Mukai transforms $\Psi$ and $\HPsi[1]$, the images of the complexes in the subcategories $\bB^{\SX}$ and $\bB^{\SY}$ behave somewhat like the images of sheaves on abelian surfaces under the  Fourier-Mukai transform (see \eqref{imagesurface}). Also as similar to \cite{MP1, MP2, Piy}, one can prove 
\begin{equation*}
\left.\begin{aligned}
 & \Psi (\tT^{\SX}_2) \subset \langle \fF^{\SY}_2 , \tT^{\SY}_2[-1]  \rangle \\        
 &  \Psi (\fF^{\SX}_2) \subset \langle \fF^{\SY}_2[-1] , \tT^{\SY}_2[-2]  \rangle
       \end{aligned}
  \ \right\},
\end{equation*}
and similar relations for $\HPsi[1]$. 
Since $\aA^{\SX} = \langle \fF^{\SX}_2[1] , \tT^{\SX}_2 \rangle$ and $\aA^{\SY} = \langle \fF^{\SY}_2[1] , \tT^{\SY}_2 \rangle$
we have $\Psi [1] (\aA^{\SX} ) = \aA^{\SY}$
as required. 

Consequently, by similar arguments as in  \cite{MP1, MP2, Piy}, one can prove the strong Bogomolov-Gieseker type inequality introduced in \cite{BMT} holds for $X$ and also $Y$.

\begin{thm}[\cite{Piy2}]
\label{equivalence3fold}
Conjecture \ref{conjequivalence} is true for abelian 3-folds. Moreover, the strong Bogomolov-Gieseker type inequality introduced in \cite{BMT} holds for abelian 3-folds.
\end{thm}
%%%%%%%%%%%%%%%%%%%%%%%%%%%%%%%%%%
%%%%%%%%%%%%%%%%%%%%%%%%%%%%%%%%%%
\section*{Acknowledgements}
The author is grateful to Arend Bayer, Tom Bridgeland, Antony Maciocia and Yukinobu Toda for very useful discussions relating to this work, and also to the organizers of the Kinosaki Symposium on Algebraic Geometry 2015. Special thanks go to Antony Maciocia for his guidance given to the author's doctoral studies.
This work is supported by the World Premier International Research Center Initiative (WPI Initiative), MEXT, Japan.
%%%%%%%%%%%%%%%%%%%%%%%%%%%%%%%%%%
%%%%%%%%%%%%%%%%%%%%%%%%%%%%%%%%%%

\providecommand{\bysame}{\leavevmode\hbox to3em{\hrulefill}\thinspace}
\providecommand{\MR}{\relax\ifhmode\unskip\space\fi MR }
% \MRhref is called by the amsart/book/proc definition of \MR.
\providecommand{\MRhref}[2]{%
  \href{http://www.ams.org/mathscinet-getitem?mr=#1}{#2}
}
\providecommand{\href}[2]{#2}

\end{document}